\documentclass[a4paper,USenglish,cleveref, autoref, thm-restate]{lipics-v2021}

\usepackage{bm}
\usepackage{array}     
\usepackage{booktabs} 
\usepackage{longtable} 
\usepackage{pdflscape} 
\usepackage{todonotes}
\usepackage{csquotes}
\usepackage{amsmath}
\usepackage{xparse}    
\usepackage{xspace}
\usepackage{mathtools}
\usepackage{upgreek}
\usepackage[inline]{enumitem}

\NewDocumentCommand{\mathOrText}{m} { \ensuremath{#1}\xspace}

\makeatletter
\def\IfEmptyTF#1%
{%
	\if\relax\detokenize{#1}\relax%
	\expandafter\@firstoftwo%
	\else%
	\expandafter\@secondoftwo%
	\fi%
}
\makeatother

\NewDocumentCommand{\functionTemplate}{m m m m o}%
{%
	\IfNoValueTF{#5}%
	{%
		\mathOrText{#1\left#2{#4}\right#3}%
	}%
	{%
		\mathOrText{#1#5#2{#4}#5#3}%
	}%
}

\DeclareDocumentCommand{\probabilisticFunctionTemplate}{m m O{} o}
{%
	\functionTemplate{#1}%
	{\lbrack}%
	{\rbrack}%
	{#2\IfEmptyTF{#3}{}{\ \IfNoValueTF{#4}{\left}{#4}\vert\ \vphantom{#2}#3\IfNoValueTF{#4}{\right.}{}}}%
	[#4]%
}

\newcommand*{\model}{Flip-Schelling-Process\xspace}
\newcommand*{\shortmodel}{FSP\xspace}
\newcommand*{\N}{\mathOrText{\mathbf{N}}}
\newcommand*{\R}{\mathOrText{\mathbf{R}}}
\newcommand*{\circlePi}{\mathOrText{\uppi}}

\newcommand*{\dist} {\mathOrText{\mathrm{dist}}}

\newcommand*{\Bin} {\mathOrText{\mathrm{Bin}}}
\newcommand*{\Multi} {\mathOrText{\mathrm{Multi}}}

\RenewDocumentCommand{\Pr}{m O{} o} {\probabilisticFunctionTemplate{\mathrm{Pr}}{#1}[#2][#3]}

\NewDocumentCommand{\E}{m O{} o} {\probabilisticFunctionTemplate{\mathrm{E}}{#1}[#2][#3]}
\NewDocumentCommand{\Var}{m O{} o} {\probabilisticFunctionTemplate{\mathrm{Var}}{#1}[#2][#3]}

\DeclareDocumentCommand{\bigO}{m o} {\functionTemplate{\mathrm{O}}{(}{)}{#1}[#2]}
\DeclareDocumentCommand{\smallO}{m o} {\functionTemplate{\mathrm{o}}{(}{)}{#1}[#2]}
\DeclareDocumentCommand{\bigTheta}{m o} {\functionTemplate{\upTheta}{(}{)}{#1}[#2]}
\DeclareDocumentCommand{\bigOmega}{m o} {\functionTemplate{\upOmega}{(}{)}{#1}[#2]}
\DeclareDocumentCommand{\smallOmega}{m o} {\functionTemplate{\upomega}{(}{)}{#1}[#2]}
\DeclareDocumentCommand{\eulerE}{o} {\mathOrText{\mathrm{e}\IfNoValueTF{#1}{}{^{#1}}}}
\DeclareDocumentCommand{\d}{o} {\mathOrText{\mathop{}\!\mathrm{d}\IfNoValueTF{#1}{}{^{#1}}}}

\DeclareMathOperator{\e}{\mathrm{e}}
\DeclareMathOperator{\sgn}{\mathrm{sgn}}

\newcommand{\uandv}{{u \cap v}}

\newcommand{\unov}{{u \setminus v}}
\newcommand{\vnou}{{v \setminus u}}
\newcommand{\nouv}{{\overline{u \cup v}}}
\newcommand{\avdegree}{\overline{\delta}}

\newcommand*{\monochrome}{M}
\newcommand*{\region}{R}
\newcommand*{\binomial}{B}
\newcommand*{\neighborhood}{N}
\newcommand*{\decisiveness}{D}

\title{The Flip Schelling Process on Random Geometric and Erd{\H o}s--Rényi Graphs} 

\titlerunning{The Flip Schelling Process on Random Graphs} 

\author
	{Thomas Bläsius}
	{Karlsruhe Institute of Technology, Karslruhe, Germany \and \url{https://scale.iti.kit.edu/people/thomasblaesius}}
	{thomas.blaesius@kit.edu}
	{}
	{} 
	
\author
	{Tobias Friedrich}
	{Hasso Plattner Institute, University of Potsdam, Potsdam, Germany \and \url{https://hpi.de/friedrich/}}
	{tobias.friedrich@hpi.de}
	{0000-0003-0076-6308}
	{} 

\author
	{Martin~S. Krejca}
	{Sorbonne Université, CNRS, LIP6, France}
	{martin.krejca@lip6.fr}
	{0000-0002-1765-1219}
	{This work was supported by the Paris Île-de-France Region.} 

\author
	{Louise Molitor}
	{Hasso Plattner Institute, University of Potsdam, Potsdam, Germany \and \url{https://hpi.de/friedrich/people/louise-molitor.html}}
	{louise.molitororc@hpi.de}
	{0000-0002-9166-9927}
	{} 

\authorrunning{T. Bläsius, T. Friedrich, M.\,S. Krejca and L. Molitor} 

\Copyright{Thomas Bläsius, Tobias Friedrich, Martin~S. Krejca and Louise Molitor} 

\ccsdesc[100]{Theory of computation~Network formation} 
\ccsdesc[300]{Theory of computation~Random network models}

\keywords{Agent-based Model, Schelling Segregation, Spin System} 

\category{} 

\relatedversion{} 

\acknowledgements{We want to thank Thomas Sauerwald for the discussions on random walks.}

\EventEditors{Nikhil Bansal and James Worrell}
\EventNoEds{2}
\EventLongTitle{48th International Colloquium on Automata, Languages, and Programming (ICALP 2021)}
\EventShortTitle{ICALP 2021}
\EventAcronym{ICALP}
\EventYear{2021}
\EventDate{July 12--16, 2021}
\EventLocation{Glasgow, United Kingdom}
\EventLogo{}
\SeriesVolume{}
\ArticleNo{}

\hideLIPIcs
\nolinenumbers

\begin{document}

	\maketitle

\begin{abstract}
	Schelling's classical segregation model gives a coherent explanation for the wide-spread phenomenon of residential segregation. We consider an agent-based saturated open-city variant, the \model\ (\shortmodel), in which agents, placed on a graph, have one out of two types and, based on the predominant type in their neighborhood, decide whether to changes their types; similar to a new agent arriving as soon as another agent leaves the vertex.
	
	We investigate the probability that an edge $\{u,v\}$ is monochrome, i.e., that both vertices $u$ and~$v$ have the same type in the \shortmodel, and we provide a general framework for analyzing the influence of the underlying graph topology on residential segregation. In particular, for two adjacent vertices, we show that a highly decisive common neighborhood, i.e., a common neighborhood where the absolute value of the difference between the number of vertices with different types is high, supports segregation and moreover, that large common neighborhoods are more decisive. 
	
	As an application, we study the expected behavior of the \shortmodel on two common random graph models with and without geometry:
	(1) For random geometric graphs, we show that the existence of an edge $\{u,v\}$ makes a highly decisive common neighborhood for $u$ and $v$ more likely. Based on this, we prove the existence of a constant $c > 0$ such that the expected fraction of monochrome edges after the \shortmodel is at least $1/2 + c$.
	(2) For Erd{\H o}s--Rényi graphs we show that large common neighborhoods are unlikely and that the expected fraction of monochrome edges after the \shortmodel is at most $1/2 + \smallO{1}$. Our results indicate that the cluster structure of the underlying graph has a significant impact on the obtained segregation strength.
\end{abstract}

\newpage

\section{Introduction}
	Residential segregation is a well-known sociological phenomenon~\cite{White86} where different groups of people tend to separate into largely homogeneous neighborhoods. Studies, e.g.~\cite{Clark86}, show that individual preferences are the driving force behind present residential patterns and bear much to the explanatory weight. Local choices therefore lead to a global phenomenon~\cite{Schelling06}. A simple model for analyzing residential segregation was introduced by Schelling~\cite{Schelling69, Schelling06} in the 1970s. In his model, two types of agents, placed on a grid, act according to the following threshold behavior, with $\tau \in (0,1)$ as the \emph{intolerance threshold}: agents are \emph{content} with their current position on the grid if at least a $\tau$-fraction of neighbors is of their own type. Otherwise they are \emph{discontent} and want to move, either via swapping with another random discontent agent or via jumping to a vacant position. Schelling demonstrated by experiments that, starting from a uniform random distribution, the described process drifts towards strong segregation, even if agents are tolerant and agree to live in mixed neighborhoods, i.e., if $\tau \leq \frac12$. Many empirical studies have been conducted to investigate the influence of various parameters on the obtained segregation, see~\cite{BW07,IEE09,fossett1998simseg,Pan07,Rogers_2011}.
	
	On the theoretical side, Schelling's model started recently gaining traction within the algorithmic game theory and artificial intelligence communities~\cite{A+19,bil_et_al:LIPIcs:2020:12684,10.5555/3398761.3398794,CLM18,E+19,elkind19,KKV20}, with focus on core game theoretic questions, where agents strategically select locations. 
	
	Henry et al.~\cite{Henry8605} described a simple model of graph clustering	motivated by Schelling where they showed that segregated graphs always emerge. Variants of the random Schelling segregation process were analyzed by a line of work that showed that residential segregation occurs with high probability~\cite{BEL16,BEL14,Bhakta14,BIK12,BIK17,You98}. 
	
	We consider an agent-based model, called the \emph{\model\ (\shortmodel)}, which can be understood as the Schelling model in a \emph{saturated open city}. In contrast to \emph{closed cities}~\cite{BEL14,BIK12,BIK17,You98}, which require fixed populations, open cities \cite{BEL15,BEL16,Bhakta14,PhysRevE.81.066120} allow residents to move away. In saturated city models, also known as voter models~\cite{durrett1993,10.2307/2244893,Liggett1999}, vertices are not allowed to be unoccupied, hence, a new agent enters as soon as one agent vacates a vertex. In general, in voter models two types of agents are placed on a graph. An agent examines their neighbors and, if a certain threshold is of another type, they change their type. Also in this model segregation is visible. There is a line of work, mainly in physics, that studies the voting dynamics on several types of graphs~\cite{https://doi.org/10.1002/rsa.20281,PhysRevE.67.026104,LIMA20083503,PhysRevE.71.016123,PhysRevE.81.011133}.
	
 	In the \shortmodel, agents have binary types. An agent is content if the fraction of agents in their neighborhood with the same type is larger $\frac12$. Otherwise, if the fraction is smaller $\frac12$, an agent is discontent and is willing to flip their type to become content. If the fraction of same type agents in their neighborhood is exactly $\frac12$, an agent flips their type with probability $\frac12$. Starting from an initial configuration where the type of each agent is chosen uniformly at random, we investigate a simultaneous-move, one-shot process and bound the number of monochrome edges, which is a popular measurement for segregation strength~\cite{CR15,freeman1978segregation}.
 	
 	Close to our model is the work by Omidvar and Franceschetti~\cite{Omidvar2018, 10.1145/3188745.3188836}, who initiated an analysis of the size of monochrome regions in the so called \emph{Schelling Spin Systems}. Agents of two different types are placed on a grid~\cite{Omidvar2018} and a geometric graph~\cite{10.1145/3188745.3188836}, respectively. Then independent and identical Poisson clocks are assigned to all agents and, every time a clock rings, the state of the corresponding agent is flipped if and only if the agent is discontent w.r.t. a certain intolerance threshold $\tau$ regarding the neighborhood size. 
 	The model corresponds to the Ising model with zero temperature with Glauber dynamics~\cite{RevModPhys.81.591,Solomon2007}. 
	
	The commonly used underlying topology for modeling the residential areas are (toroidal) grid graphs~\cite{bil_et_al:LIPIcs:2020:12684,BIK17,Omidvar2018}, regular graphs~\cite{bil_et_al:LIPIcs:2020:12684,CLM18,E+19}, paths~\cite{bil_et_al:LIPIcs:2020:12684,KKV20}, cycles~\cite{BEL15,BEL18b,BEL14,BIK12,You98} and trees~\cite{A+19,bil_et_al:LIPIcs:2020:12684,elkind19,KKV20}.
	Considering the influence of the given topology that models the residential area regarding, e.g., the existence of stable states and convergence behavior leads to phenomena like non-existence of stable states~\cite{E+19,elkind19}, non-convergence to stable states~\cite{bil_et_al:LIPIcs:2020:12684,CLM18,E+19}, and high-mixing times~\cite{Bhakta14,Gerhold08}.
	
	To avoid such undesirable characteristics, we suggest to investigate \emph{random geometric graphs}~\cite{PenroseRandomGeometricGraphs}, like in~\cite{10.1145/3188745.3188836}. Random geometric graphs demonstrate, in contrast to other random graphs without geometry, such as \emph{Erd{\H o}s--Rényi graphs}~\cite{Erdos1959,gilbert1959}, community structures, i.e., densely connected clusters of vertices. An effect observed by simulating the \shortmodel\ is that the fraction of monochrome edges is significantly higher in random geometric graphs compared to Erd{\H o}s--Rényi graphs, where the fraction stays almost stable around $\frac12$, cf. Fig~\ref{fig:eval}.	
	
	\begin{figure}
    	\centering
        \small
		\input{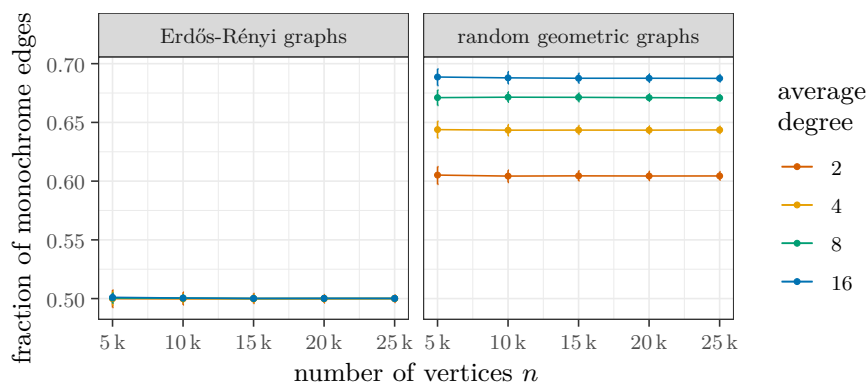}
		\caption{The fraction of monochrome edges after the Flip-Schelling-Process (\shortmodel) in Erd{\H o}s--Rényi graphs and random geometric graphs for different graph sizes (number of vertices~$n$) and different expected average degrees. Each data point
		shows the average over 1000 generated graphs with one simulation of the FSP per graph. The error bars show the interquartile range, i.e., $50\,\%$ of the measurements lie between the top and bottom end of the error bar.}
		 \label{fig:eval}
	\end{figure}

	We set out for rigorously proving this phenomenon. In particular, we prove for random geometric graphs that there exists a constant $c$ such that, given an edge $\{u,v\}$, the probability that $\{u,v\}$ is monochrome is lower-bounded by $\frac12 + c$, cf. \Cref{thm:final-result}. In contrast, we show for Erd{\H o}s--Rényi graphs that segregation is not likely to occur and the probability that $\{u,v\}$ is monochrome is upper-bounded by  $\frac12 + \smallO{1}$, cf. {\Cref{Thm:ER}.
	
	We introduce a general framework to deepen the understanding of the influence of the underlying given topology on residential segregation. To this end, we first show that a highly decisive common neighborhood supports segregation, cf.~\Cref{sec:monochrom-edges-via}. In particular, we provide a lower bound that an edge $\{u,v\}$ is monochrome based on the probability that the difference between the majority and the minority regarding both types in the common neighborhood, i.e., the number of agents which are adjacent to $u$ and~$v$, is larger compared to their exclusive neighborhoods, i.e., the number of agents which are adjacent to either $u$ or~$v$.
	Next, we show that large sets are more decisive, cf.~\Cref{sec:large-neighb-are}. This implies that a large common neighborhood, compared to the exclusive neighborhood, is likely to be more decisive, i.e., makes it more likely that the absolute value of the difference between the number of different types in the common neighborhood is larger than in the exclusive neighborhoods. These considerations hold for arbitrary graphs. Hence, we reduce the question concerning a lower bound for the fraction of monochrome edges in the \shortmodel\ to the probability that, given $\{u,v\}$, the common neighborhood of $u$ and $v$ is larger than the exclusive neighborhoods of $u$ and $v$, respectively.
	
	For random geometric graphs, we prove that a large geometric region, i.e., the intersecting region that are formed by intersecting disks, leads to a large vertex set, cf.~\Cref{subsec:neighborhood_region}, and that random geometric graphs have enough edges that have sufficiently large intersecting regions, cf.~\Cref{sec:many-edges-have}, such that segregation is likely to occur. In contrast, for Erd{\H o}s--Rényi graphs, we show that the common neighborhood between two vertices $u$ and $v$ is with high probability empty and therefore segregation is not likely to occur, cf.~\Cref{sec:erdosrenyi}. 
	
	Overall, we shed light on the influence of the structure of the underlying graph and discovered the significant impact of the community structure as an important factor on the obtained segregation strength.
	We reveal for random geometric graphs that already after one round a provable tendency is apparent and a strong segregation occurs.

\section{Model and Preliminaries}

Let $G = (V,E)$ be an unweighted and undirected graph, with vertex set $V$ and edge set $E$. For any vertex $v \in V$, we denote the \emph{neighborhood} of $v$ in $G$ by $N_v = \{u \in V\colon \{u,v\} \in E\}$ and the degree of $v$ in~$G$ by $\delta_v = |N_v|$.
We consider \emph{random geometric graphs} and \emph{Erd{\H o}s--Rényi graphs} with a total of  $n \in \N^+$ vertices and an \emph{expected average degree} $\avdegree > 0$. 

For a given $r \in \R^+$, a random geometric graph $G \sim \mathcal{G}(n, r)$ is obtained by distributing~$n$ vertices uniformly at random in some geometric ground space and connecting vertices $u$ and~$v$ if and only if $\dist(u, v) \le r$.  We use a two-dimensional toroidal Euclidean space with total area $1$ as ground space.  More formally, each vertex $v$ is assigned a point $(v_1, v_2) \in [0, 1]^2$ and 
the distance between $u = (u_1, u_2)$ and $v$ is $\dist(u, v) = \sqrt{|u_1 - v_1|_{\circ}^2 + |u_2 - v_2|_{\circ}^2}$  for $|u_i - v_i|_{\circ} = \min\{|u_i - v_i|, 1 - |u_i - v_i|\}$.  We note that using a torus instead of, e.g., a unit square, has the advantage that we do not have to consider edge cases, for vertices that are close to the boundary.  In fact, a disk of radius $r$ around any point has the same area $\circlePi r^2$.  As every vertex $v$ is connected to all vertices in the disk of radius $r$ around it, its expected average degree is $\avdegree = (n - 1) \circlePi r^2$.
	
For a given $p \in [0,1]$, let $\mathcal{G}(n,p)$ denote an Erd{\H o}s--Rényi graph. Each edge $\{u,v\}$ is included with probability $p$, independently from every other edge. It holds that $\avdegree = (n-1)p$.

Consider two different vertices $u$ and $v$. Let $N_\uandv \coloneqq |N_u \cap N_v|$ be the number of vertices in the \emph{common neighborhood}, let $N_\unov \coloneqq |N_u \setminus N_v|$ be the number of vertices in the \emph{exclusive neighborhood} of $u$, and let $N_\vnou \coloneqq  |N_v \setminus N_u|$ be the number of vertices in the exclusive neighborhood of $v$. Furthermore, with~$N_\nouv \coloneqq |V \setminus (N_u \cup N_v)|$, we denote the number of vertices that are neither adjacent to $u$ nor to $v$. 

Let $G$ be a graph where each vertex represents an agent of type $t^+$ or $t^-$.  The type of each agent is chosen independently and uniformly at random.  An edge $\{u,v\}$ \emph{monochrome} if and only if $u$ and $v$ are of the same type. The \emph{\model} (\shortmodel) is defined as follows: an agent $v$ whose type is aligned with the type of more than $\delta_v/2$ of their neighbors keeps their type. If more than $\delta_v/2$ neighbors have a different type, then agent $v$ changes their type. In case of a tie, i.e., if exactly $\delta_v/2$ neighbors have a different type, then $v$ changes their type with probability $\frac12$. \shortmodel is a simultaneous-move, one-shot process, i.e., all agents make their decision at the same time and, moreover, only once. 

For $x,y \in \N$, we define $[x..y] = [x,y] \cap \N$ and for $x \in \N$, we define $[x] = [1..x]$.

\subsection{Useful Technical Lemmas}

In this section, we state several lemmas that we will use in order to prove our results in the next sections.

\begin{lemma}
	\label{lem:dominating_binomial_distribution}
	Let $X \sim \Bin(n, p)$ and $Y \sim \Bin(n, q)$ with $p \ge q$.
	Then $\Pr{X \ge Y} \ge \frac{1}{2}$.
\end{lemma}
\begin{proof}
	Let $Y_1, \dots, Y_n$ be the individual Bernoulli trials for $Y$,
	i.e., $Y = \sum_{i \in [n]} Y_i$.  Define new random variables
	$Y_1', \dots, Y_n'$ such that $Y_i = 1$ implies $Y_i' = 1$ and if
	$Y_i = 0$, then $Y_i' = 1$ with probability $(p - q) / (1 - q)$ and $Y_i' = 0$
	otherwise.  Note that for each individual $Y_i'$,
	we have $Y_i' = 1$ with probability $p$, i.e.,
	$Y' = \sum_{i \in [n]} Y_i' \sim \Bin(n, p)$.  Moreover, as
	$Y' \ge Y$ for every outcome, we have
	$\Pr{X \ge Y} \ge \Pr{X \ge Y'}$.  It remains to show that
	$\Pr{X \ge Y'} \ge \frac{1}{2}$.
	
	As $X$ and $Y'$ are equally distributed, we have
	$\Pr{X \ge Y'} = \Pr{X \le Y'}$.  Moreover, as one of the two
	inequalities holds in any event, we get
	$\Pr{X \ge Y'} + \Pr{X \le Y'} \ge 1$, and thus equivalently
	$2\Pr{X \ge Y'} \ge 1$, which proves the claim.
\end{proof}

\begin{lemma}
	\label{lem:maximum_of_a_binomial_distribution}
	Let $n \in \N^+$, $p \in [0, 1)$, and let $X \sim \Bin(n, p)$.
	Then, for all $i \in [0.. n] $, it holds that $\Pr{X = i} \leq \Pr{X = \lfloor p(n + 1) \rfloor}$.
\end{lemma}

\begin{proof}
	We interpret the distribution of~$X$ as a finite series and consider the sign of the differences of two neighboring terms.
	A maximum of the distribution of~$X$ is located at the position at which the difference switches from positive to negative.
	To this end, let $b\colon [0, n - 1] \to [-1, 1]$ be defined such that, for all $i \in [0, n - 1] \cap \N$, it holds that
	\[
	b(d) = \binom{n}{d + 1} p^{d + 1} (1 - p)^{n - d - 1} - \binom{n}{d} p^d (1 - p)^{n - d}.
	\]
	We are interested in the sign of~$b$.
	In more detail, for any $d \in [0, n - 2] \cap \N$, if $\sgn\big(b(d)\big) \geq 0$ and $\sgn\big(b(d + 1)\big) \leq 0$, then $d + 1$ is a local maximum.
	If the sign is always negative, then there is a global maximum in the distribution of~$X$ at position~$0$.
	
	In order to determine the sign of~$b$, for all $i \in [0.. n - 1] $, we rewrite
	\begin{align*}
		b(i) &= \frac{n!}{i! (n - i - 1)!} p^i (1 - p)^{n - i - 1} \frac{p}{i + 1} - \frac{n!}{i! (n - i - 1)!} p^d (1 - p)^{n - i - 1} \frac{1 - p}{n - i}\\
		&= \frac{n!}{i! (n - i - 1)!} p^i (1 - p)^{n - i - 1} \left(\frac{p}{i + 1} - \frac{1 - p}{n - i}\right).
	\end{align*}
	Since the term $n! p^i (1 - p)^{n - i - 1}$ is always non-negative, the sign of $b(i)$ is determined by the sign of $p/(i + 1) - (1 - p)/(n - i)$.
	Solving for~$i$, we get that
	\[
		\frac{p}{i + 1} - \frac{1 - p}{n - i} \geq 0 \Leftrightarrow i \leq p(n + 1) - 1.
	\]
	Note that $p(n + 1) - 1$ may not be integer.
	Further note that the distribution of~$X$ is unimodal, as the sign of~$b$ changes at most once.
	Thus, each local maximum is also a global maximum.
	As discussed above, the largest value $d \in [0, n - 2] \cap \N$ such that $\sgn\big(b(d)\big) \geq 0$ and $\sgn\big(b(d + 1)\big) \leq 0$ then results in a global maximum at position $d + 1$.
	Since~$d$ needs to be integer, the largest value that satisfies this constraint is $\lfloor p(n + 1) - 1 \rfloor$.
	If the sign of~$b$ is always negative ($p \leq 1/(n + 1)$), then the distribution of~$X$ has a global maximum at~$0$, which is also satisfied by $\lfloor p(n + 1) - 1 \rfloor + 1$, which concludes the proof.
\end{proof}

\begin{theorem}[Stirling's Formula~{\cite[page~54]{FellerAnIntroductionToProbabilityTheory}}]
	\label{thm:stirling_approximation}
	For all $n \in \N^+$, it holds that
	\[
		\sqrt{2 \circlePi} n^{n+1/2}\e^{-n} \cdot \e^{\left(12n+1\right)^{-1}} < n! < \sqrt{2 \circlePi} n^{n+1/2}\e^{-n} \cdot \e^{\left(12n\right)^{-1}}.
	\]
\end{theorem}

\begin{corollary}
	\label{cor:stirling_approximation}
	For all $n \geq 2$ with $n \in \N$, it holds that
	\begin{align}
		\label{eq:stirling_inequality_lowerbound}
		n! &> \sqrt{2 \circlePi} n^{n+1/2}\e^{-n}   \textrm{ and}\\
		\label{eq:stirling_inequality_upperbound}
		n! &<\e n^{n+1/2}\e^{-n}.
	\end{align}
\end{corollary}

\begin{proof}
	For both inequalities, we aim at using \Cref{thm:stirling_approximation}.
	
	\noindent \cref{eq:stirling_inequality_lowerbound}: Note that $\e^{\left(12n+1\right)^{-1}} > 1$, since $\frac{1}{12n+1} > 0$. Hence,
	\[
	\sqrt{2 \circlePi} n^{n+1/2}\e^{-n}   < \sqrt{2 \circlePi} n^{n+1/2}\e^{-n} \cdot \e^{\left(12n+1\right)^{-1}}.
	\]
	
	\noindent \cref{eq:stirling_inequality_upperbound}: We prove this case by showing that 
	\begin{align}
		\label{eq:proof_stirling_inequality_lowerbound}
		\sqrt{2 \circlePi} \e^{\left(12n\right)^{-1}} < \e.
	\end{align}
	Note, that $\e^{\left(12n\right)^{-1}}$ is strictly decreasing. Hence, we only have to check whether \cref{eq:proof_stirling_inequality_lowerbound} holds for $n = 2$.
	\[
	\sqrt{2 \circlePi} \e^{\left(12n\right)^{-1}} \leq \sqrt{2 \circlePi}  \e^{\frac{1}{24}} < 2.7 < \e. \qedhere
	\]
\end{proof}

\begin{lemma}
	\label{lem:obvious-observation-on-conditional-probabilities}
	Let $A$, $B$, and $C$ be random variables such that $\Pr{A > C \land B > C} > 0$ and $\Pr{A > C \land B \leq C} > 0$.  Then
	$\Pr{A > B \land A > C} \ge \Pr{A > B} \cdot \Pr{A>C}$.
\end{lemma}
\begin{proof}
	Using the definition of conditional probability, we obtain
	\[
		\Pr{A > B \land A > C} = \Pr{A>B}[A>C] \cdot \Pr{A>C}.
	\]
	Hence, we are left with bounding $\Pr{A > B}[A > C] \ge \Pr{A > B}$.
	Partitioning the sample space into the two events $B > C$ and
	$B \le C$ and using the law of total probability, we obtain
	\begin{align*}
		\Pr{A > B \mid A > C}
		=\,& \Pr{B > C}[A > C] \cdot \Pr{A > B \mid A > C \wedge B > C} \\
		+\,& \Pr{B \le C}[A > C] \cdot \Pr{A > B \mid A > C \wedge B \le C}.
	\end{align*}
	Note that the condition $A > C \wedge B \le C$ already implies
	$A > B$ and thus the last probability equals to~$1$.  Moreover,
	using the definition of conditional probability, we obtain
	\begin{align*}
		\Pr{A > B \mid A > C}
		=\ &\Pr{B > C}[A > C] \cdot \frac{\Pr{A > B \wedge A > C \wedge B > C}}{\Pr{A > C \wedge B > C}}\\
		&+ \Pr{B \le C}[A > C].
	\end{align*}
	Using that $\Pr{B > C}[A > C] \ge \Pr{A > C \wedge B > C}$, that
	$A > B \wedge B > C$ already implies $A > C$, that $\Pr{B \le C}[A > C] \ge
	\Pr{A > B \wedge B \le C}$, and finally the law of total
	probability, we obtain
	\begin{align*}
		\Pr{A > B \mid A > C}
		&\ge \Pr{A > B \wedge A > C \wedge B > C} + \Pr{B \le C}[A > C]\\
		&= \Pr{A > B \wedge B > C} + \Pr{B \le C}[A > C]\\
		&\ge \Pr{A > B \wedge B > C} + \Pr{A > B \wedge B \le C}\\
		&=\Pr{A > B}.   \qedhere
	\end{align*}
\end{proof}

\section{Monochrome Edges in Geometric Random Graphs}
\label{sec:monochr-edges-geom}

In this section, we prove the following main theorem.

\begin{theorem}
  \label{thm:final-result}
  Let $G \sim \mathcal{G}(n,r)$ be a random geometric graph with
  expected average degree $\avdegree = \smallO{\sqrt{n}}$.  The expected
  fraction of monochrome edges after the FSP is at least
  \begin{equation*}
    \frac{1}{2} +
    \frac{9}{800} \cdot
    \left(\frac{1}{2} - \frac{1}{\sqrt{2 \pi \lfloor \avdegree / 2\rfloor}}\right)^2 \cdot
    \left(1 - \eulerE^{-\avdegree/2} \left(1 + \frac{\avdegree}{2}\right)\right) \cdot
    (1 - \smallO{1}).
  \end{equation*}
\end{theorem}
Note that the bound in Theorem~\ref{thm:final-result} is bounded away
from $\frac{1}{2}$ for all $\avdegree \ge 2$.  Moreover, the two
factors depending on $\avdegree$ go to $\frac{1}{2}$ and $1$,
respectively, for a growing $\avdegree$.

Given an edge $\{u, v\}$, we prove the above lower bound on the
probability that $\{u, v\}$ is monochrome in the following four steps.
\begin{enumerate*}
\item For a vertex set, we introduce the concept of
  \emph{decisiveness} that measures how much the majority is ahead of
  the minority in the \shortmodel.  With this, we give a lower
  bound on the probability that $\{u, v\}$ is monochrome based on the
  probability that the common neighborhood of $u$ and $v$ is more
  decisive than their exclusive neighborhoods.
\item We show that large neighborhoods are likely to be more decisive
  than small neighborhoods.  To achieve this, we give bounds on the
  likelihood that two similar random walks behave differently.  This
  step reduces the question of whether the common neighborhood is more
  decisive than the exclusive neighborhoods to whether the former is
  larger than the latter.
\item Turning to geometric random graphs, we show that the common
  neighborhood is sufficiently likely to be larger than the exclusive
  neighborhoods if the geometric region corresponding to the former is
  sufficiently large.  We do this by first showing that the actual
  distribution of the neighborhood sizes is well approximated by
  independent random variables that follow binomial distributions.
  Afterwards, we give the desired bounds for these random variables.
\item We show that the existence of the edge $\{u, v\}$ in the
  geometric random graph makes it sufficiently likely that the
  geometric region hosting the common neighborhood of $u$ and $v$ is
  sufficiently large.
\end{enumerate*}

\subsection{Monochrome Edges via Decisive Neighborhoods}
\label{sec:monochrom-edges-via}

Let $\{u, v\}$ be an edge of a given graph.  To formally define the above mentioned decisiveness, let $N_\uandv^+$ and  $N_\uandv^-$ be the number of vertices in the common neighborhood of $u$ and $v$ that are occupied by  agents of type $t^+$ and $t^-$, respectively.
Then $D_\uandv \coloneqq |N_\uandv^+ - N_\uandv^-|$ is the \emph{decisiveness} of the common neighborhood of $u$ and $v$.  Analogously, we define $D_\unov$ and $D_\vnou$ for the exclusive neighborhoods of $u$ and $v$, respectively. 

The following theorem bounds the probability for $\{u, v\}$ to be monochrome based on the probability that the common neighborhood is more decisive than each of the exclusive neighborhoods.

\begin{theorem}
  \label{lem:decision_tree}
  In the \shortmodel, let $\{u, v\} \in E$ be an edge and let $\decisiveness$ be the event $\{D_\uandv > D_\unov \wedge D_\uandv > D_\vnou\}$. Then $\{u, v\}$ is monochrome with probability at least $1 / 2 + \Pr{\decisiveness} /2$.
\end{theorem}
\begin{proof} 
  If $\decisiveness$ occurs, then the types of~$u$ and~$v$ after the
  \shortmodel coincide with the predominant type before the \shortmodel in the shared
  neighborhood. Thus, $\{u,v\}$ is monochrome.
  
  Otherwise, assuming $\overline{\decisiveness}$, w.l.o.g., let $D_\uandv \le D_\unov$ and assume the type of $v$ has
  already been determined.  If $D_\uandv = D_\unov$, then $u$ chooses
  a type uniformly at random, which coincides with the type of $v$
  with probability $\frac{1}{2}$.  Otherwise, $D_\uandv < D_\unov$ and
  thus $u$ takes the type that is predominant in $u$'s exclusive
  neighborhood, which is $t^+$ and $t^-$ with probability
  $\frac{1}{2}$, each.  Moreover, this is independent from the type of
  $v$ as $v$'s neighborhood is disjoint to $u$'s exclusive
  neighborhood.

  Thus, for the event $\monochrome$ that $\{u, v\}$ is monochrome, we
  get $\Pr{\monochrome \mid \decisiveness} = 1$ and
  $\Pr{\monochrome \mid \overline \decisiveness} = \frac{1}{2}$.
  Hence, we get
  $\Pr{\monochrome} > \Pr{\decisiveness} + \frac12
  \left( 1 - \Pr{\decisiveness} \right) = \frac12 + \Pr{\decisiveness}
  /2$.
\end{proof}
	
\subsection{Large Neighborhoods are More Decisive}
\label{sec:large-neighb-are}

The goal of this section is to reduce the question of how decisive a
neighborhood is to the question of how large it is. To be more
precise, assume we have a set of vertices of size $a$ and give each
vertex the type $t^+$ and $t^-$, each with probability $\frac{1}{2}$.
Let $A_i$ for $i \in [a]$ be the random variable that takes the value
$+1$ and $-1$ if the $i$-th vertex in this set has type $t^+$ and
$t^-$, respectively.  Then, for $A = \sum_{i \in [a]} A_i$, the
decisiveness of the vertex set is $|A|$.  In the following, we study
the decisiveness $|A|$ depending on the size $a$ of the set.  Note
that this can be viewed as a random walk on the integer line: Starting
at $0$, in each step, it moves one unit either to the left or to the
right with equal probabilities.  We are interested in the distance
from $0$ after $a$ steps.

Assume for the vertices $u$ and $v$ that we know that $b$ vertices lie in the common neighborhood and
$a$ vertices lie in the exclusive neighborhood of $u$.  Moreover, let
$A$ and $B$ be the positions of the above random walk after $a$ and
$b$ steps, respectively.  Then the event $D_\uandv > D_\unov$ is
equivalent to $|B| > |A|$.  Motivated by this, we study the
probability of $|B| > |A|$, assuming $b \ge a$.  The core difficulty
here comes from the fact that we require $|B|$ to be strictly larger
than $|A|$.  Also note that $a + b$ corresponds to the degree of $u$ in
the graph.  Thus, we have to study the random walks also for small
numbers of $a$ and $b$.  We note that all results in this section are
independent from the specific application to the \shortmodel, and thus
might be of independent interest.

Before we give a lower bound on the probability that $|B| > |A|$, we
need the following technical lemma.  It states that doing
more steps in the random walk only makes it more likely to deviate
further from the starting position.

\begin{lemma}
  \label{lem:obvious-lemma-rand-walk}
  For $i \in [a]$ and $j \in [b]$ with $0 \le a \le b$, let $A_i$ and
  $B_j$ be independent random variables that are $-1$ and $1$ each
  with probability $\frac{1}{2}$.  Let $A = \sum_{i\in [a]} A_i$ and
  $B = \sum_{j \in [b]} B_j$.  Then $\Pr{|A| < |B|} \ge \Pr{|A| > |B|}$.
\end{lemma}
\begin{proof}
  Let $\Delta_k$ be the event that $|B| - |A| = k$.  First note that
  \begin{equation*}
    \Pr{|A| < |B|} = \sum_{k \in [b]} \Pr{\Delta_k}
    \quad \text{ and }\quad
    \Pr{|A| > |B|} = \sum_{k \in [a]} \Pr{\Delta_{-k}}. 
  \end{equation*}
  To prove the statement of the lemma, it thus suffices to prove the
  following claim.
  \begin{claim}
    \label{claim:obvious-lemma-rand-walk-claim}
    For $k \ge 0$, $\Pr{\Delta_k} \ge \Pr{\Delta_{-k}}$.
  \end{claim}
  We prove this claim via induction on $b - a$. 
  For the base case $a = b$, $A$ and $B$ are equally distributed and thus
  Claim~\ref{claim:obvious-lemma-rand-walk-claim} clearly holds.

  For the induction step, let $B^+$ be the random variable that takes
  the values $B + 1$ and $B - 1$ with probability $\frac{1}{2}$ each.
  Note that $B^+$ represents the same type of random walk as $A$ and
  $B$ but with $b + 1$ steps.  Moreover $B^+$ is coupled with $B$ to
  make the same decisions in the first $b$ steps.  Let $\Delta^+_k$ be
  the event that $|B^+| - |A| = k$.  It remains to show that
  Claim~\ref{claim:obvious-lemma-rand-walk-claim} holds for these
  $\Delta^+_{k}$.
  For this, first note that the claim trivially holds for $k = 0$.
  For $k \ge 1$, we can use the definition of $\Delta^+_k$ and the
  induction hypothesis to obtain
  \begin{align*}
    \Pr{\Delta^+_{k}}
    &= \frac{\Pr{\Delta_{k-1}}}{2} +  \frac{\Pr{\Delta_{k+1}}}{2}\\
    &\ge \frac{\Pr{\Delta_{-k+1}}}{2} + \frac{\Pr{\Delta_{-k-1}}}{2} = \Pr{\Delta^+_{-k}}. \qedhere
  \end{align*}
\end{proof}
Using \Cref{lem:obvious-lemma-rand-walk}, we can now prove the following general bound for the
probability that $|A| < |B|$, depending on certain probabilities for
binomially distributed variables.

\begin{lemma}
  \label{lem:random_walk}
  For $i \in [a]$ and $j \in [b]$ with $0 \le a \le b$, let $A_i$ and
  $B_j$ be independent random variables that are $-1$ and $1$ each
  with probability $\frac{1}{2}$.  Let $A = \sum_{i\in [a]} A_i$ and
  $B = \sum_{j \in [b]} B_j$.  Moreover, let
  $X \sim \Bin(a, \frac{1}{2})$, $Y \sim \Bin(b, \frac{1}{2})$, and
  $Z \sim \Bin(a + b, \frac{1}{2})$.  Then
  \begin{equation*}
    \Pr{|A| < |B|} \ge \frac{1}{2}
    - \Pr{Z = \frac{a + b}{2}}
    + \frac{\Pr{X = \frac{a}{2}} \cdot \Pr{Y = \frac{b}{2}}}{2}.
  \end{equation*}
\end{lemma}
\begin{proof}
  Using that $\Pr{|A| < |B|} \ge \Pr{|A| > |B|}$ (see
  Lemma~\ref{lem:obvious-lemma-rand-walk}), we obtain
  \begin{align}
    \notag
    &&\Pr{|A| < |B|} +  \Pr{|A| > |B|} +  \Pr{|A| = |B|} &= 1\\
    \notag
    &\Rightarrow& 2\Pr{|A| < |B|} + \Pr{|A| = |B|} &\ge 1\\
    \label{eq:random_walk_1}
    &\Leftrightarrow& \Pr{|A| < |B|} &\ge \frac{1}{2} -  \frac{\Pr{|A| = |B|}}{2}.
  \end{align}
  Thus, it remains to give an upper bound for $\Pr{|A| = |B|}$.

  Using the inclusion--exclusion principle and the fact that $B$ is
  symmetric around $0$, i.e., $\Pr{B = x} = \Pr{B = -x}$ for any $x$,
  we obtain
  \begin{align}
    \notag
    \Pr{|A| = |B|}
    &= \Pr{A = B \vee A = -B}\\
    \notag
    &= \Pr{A = B} + \Pr{A = -B} - \Pr{A = B = 0}\\
    \label{eq:random_walk_2}
    &= 2\Pr{A = -B} - \Pr{A = B = 0}.
  \end{align}
  We estimate $\Pr{A = -B}$ and $\Pr{A = B = 0}$ using bounds for
  binomially distributed variables.  To this end, define new random
  variables $X_i = \frac{A_i + 1}{2}$ for $i \in [a]$ and let
  $X = \sum_{i \in [a]} X_i$.  Note that the $X_i$ are independent and
  take values $0$ and $1$, each with probability $\frac{1}{2}$.  Thus,
  $X \sim \Bin(a, \frac{1}{2})$.  Moreover, $A = 2 X - a$.
  Analogously, we define $Y$ with $Y \sim \Bin(b, \frac{1}{2})$ and
  $B = 2 Y - b$.  Note that $X$ and $Y$ are independent and thus
  $Z = X + Y \sim \Bin(a + b, \frac{1}{2})$.  With this, we get
  \begin{align*}
    \Pr{A = -B}
    &= \Pr{2X - a = - 2Y + b}
      = \Pr{Z = \frac{a + b}{2}} \text{, and}\\
    \Pr{A = B = 0}
    &= \Pr{A = 0} \cdot \Pr{B = 0}
      = \Pr{X = \frac{a}{2}} \cdot \Pr{Y = \frac{b}{2}}.
  \end{align*}
  This, together with \Cref{eq:random_walk_1,eq:random_walk_2} yield the claim.
\end{proof}
The bound in Lemma~\ref{lem:random_walk} becomes worse for smaller
values of $a$ and $b$.  Considering this worst case, we obtain the
following specific bound.

\begin{theorem}
  \label{thm:random-walk}
  For $i \in [a]$ and $j \in [b]$ with $0 \le a \le b$, let $A_i$ and
  $B_j$ be independent random variables that are $-1$ and $1$ each
  with probability $\frac{1}{2}$.  Let $A = \sum_{i\in [a]} A_i$ and
  $B = \sum_{j \in [b]} B_j$. If $a = b = 0$ or $a = b = 1$, then
  $\Pr{|A| < |B|} = 0$.  Otherwise $\Pr{|A| < |B|} \ge \frac{3}{16}$.
\end{theorem}
\begin{proof}
  Clearly, if $a = b = 0$, then $A = B = 0$ and thus
  $\Pr{|A| < |B|} = 0$.  Similarly, if $a = b = 1$, then
  $|A| = |B| = 1$ and thus $\Pr{|A| < |B|} = 0$.  For the remainder,
  assume that neither $a = b = 0$ nor $a = b = 1$, and let $X$,
  $Y$, and $Z$ be defined as in Lemma~\ref{lem:random_walk}, i.e., $X \sim \Bin(a, \frac{1}{2})$, $Y \sim \Bin(b, \frac{1}{2})$, and
  $Z \sim \Bin(a + b, \frac{1}{2})$.
  
  If $a + b$ is odd, then $\Pr{Z = \frac{a + b}{2}} = 0$.  Thus, by
  Lemma~\ref{lem:random_walk}, we have
  $\Pr{|A| < |B|} \ge \frac{1}{2}$.
  If $a + b$ is even and $a + b \ge 6$, then
  \begin{equation*}
    \Pr{Z = \frac{a + b}{2}} =
    {a + b \choose \frac{a + b}{2}} \left(\frac{1}{2}\right)^{a + b} \le
    {6 \choose 3} \left(\frac{1}{2}\right)^{6} =
    \frac{5}{16}.
  \end{equation*}
  Hence, by Lemma~\ref{lem:random_walk}, we have
  $\Pr{|A| < |B|} \ge \frac{1}{2} - \frac{5}{16} = \frac{3}{16}$.

  If $a + b < 6$ (and $a + b$ even), there are four cases:
  $a = 0, b = 2$; $a = 0, b = 4$; $a = 1, b = 3$; $a = 2, b = 2$.
  If $a = 0$ and $b = 2$, then $A = 0$ with probability $1$ and
  $|B| = 2$ with probability $\frac{1}{2}$.  Thus,
  $\Pr{|A| < |B|} = \frac{1}{2}$.
  If $a = 0$ and $b = 4$, then $|A| < |B|$ unless $B = 0$.  As
  $\Pr{B = 0} = {4 \choose 2} \cdot (\frac{1}{2})^4 = \frac{3}{8}$, we
  get $\Pr{|A| < |B|} = 1 - \frac{3}{8} = \frac{5}{8}$.
  If $a = 1$ and $b = 3$, then $|A| = 1$ with probability $1$ and
  $|B| = 3$ with probability $\frac{1}{4}$ (either
  $B_1 = B_2 = B_3 = 1$ or $B_1 = B_2 = B_3 = -1$). Thus,
  $\Pr{|A| < |B|} = \frac{1}{4}$.
  If $a = b = 2$, then $|A| = 0$ with probability~$\frac{1}{2}$ and
  $|B| = 2$ with probability $\frac{1}{2}$.  Thus
  $\Pr{|A| < |B|} = \frac{1}{4}$.

  We note that the bound of $\Pr{|A| < |B|} = \frac{3}{16}$ is tight
  for $a = b = 3$.
\end{proof}

\subsection{Large Common Regions Yield Large Common Neighborhoods}
\label{subsec:neighborhood_region}

To be able to apply Theorem~\ref{thm:random-walk} to an edge
$\{u, v\}$, we need to make sure that the size of their common
neighborhood (corresponding to $b$ in the corollary) is at least the
size of the exclusive neighborhoods (corresponding to $a$ in the
corollary).  In the following, we give bounds for the probability that
this happens.  Note that this is the first time we actually take the
graph into account.  Thus, all above considerations hold for arbitrary
graphs.

Recall that we consider random geometric graphs $\mathcal{G}(n,r)$ and $u$ and $v$ are each connected to all vertices that lie
within a disk of radius $r$ around them.  As $u$ and $v$ are adjacent,
their disks intersect, which separates the ground space into four
regions; cf. Fig~\ref{fig:intersection_neighborhood}. Let $R_\uandv$ be the intersection of the two disks.
Let $R_\unov$ be the set of points that lie in the disk of~$u$ but not
in the disk of $v$, and analogously, let $R_\vnou$ be the disk of $v$
minus the disk of $u$.  Finally, let $R_\nouv$ the set of points
outside both disks.  Then, each of the $n - 2$ remaining vertices ends
up in exactly one of these regions with a probability equal to the
corresponding measure.  Let $p = \mu(R_\uandv)$ and
$q = \mu(R_\unov) = \mu(R_\vnou)$ be the probabilities for the common
and exclusive regions, respectively.  The probability for $R_\nouv$ is
then $1 - p - 2q$.

We are now interested in the sizes $N_\uandv$, $N_\unov$, and
$N_\vnou$ of the common and the exclusive neighborhoods, respectively.
As each of the $n - 2$ remaining vertices ends up in $R_\uandv$ with
probability $p$, we have $R_\uandv \sim \Bin(n - 2, p)$.  For
$N_\unov$ and $N_\vnou$, we already know that $v$ is a neighbor of $u$
and vice versa.  Thus, $(N_\unov - 1) \sim \Bin(n - 2, q)$ and
$(N_\vnou - 1) \sim \Bin(n - 2, q)$.  Moreover, the three random
variables are not independent, as each vertex lies in only exactly one
of the four regions, i.e., $N_\uandv$, $(N_\unov - 1)$,
$(N_\vnou - 1)$, and the number of vertices in neither neighborhood
together follow a multinomial distribution $\Multi(n - 2, \boldsymbol{p})$
with $\boldsymbol{p} = \left(p, q, q, 1-p-2q\right)$.

The following lemma shows that these dependencies are small if $p$ and
$q$ are sufficiently small.  This lets us assume that $N_\uandv$,
$(N_\unov - 1)$, $(N_\vnou - 1)$ are independent random variables
following binomial distributions if the expected average degree is not
too large.

\begin{lemma}
  \label{lem:multi_vs_binom}
  Let $X = \left(X_1, X_2, X_3, X_4\right) \sim
  \Multi\left(n,\boldsymbol{p}\right)$ with
  $\boldsymbol{p} = \left(p, q, q, 1-p-2q\right)$.  Then there exist
  independent random variables $Y_1 \sim \Bin\left(n,p\right)$,
  $Y_2 \sim \Bin\left(n,q\right)$ and
  $Y_3 \sim \Bin\left(n,q\right)$ such that
  $\Pr{\left(X_1, X_2, X_3\right) = \left(Y_1, Y_2,
      Y_3\right)} \geq 1 - 3 n\cdot \max(p,q)^2$.
\end{lemma}
\begin{proof}
  Let $Y_1 \sim \Bin\left(n,p\right)$, and
  $Y_2, Y_3 \sim \Bin\left(n,q\right)$ be independent random
  variables.  We define the event $\binomial$ to hold, if each of the
  $n$ individual trials increments at most one of the random variables
  $Y_1$, $Y_2$, or $Y_3$.  More formally, for $i \in [3]$ and
  $j \in [n]$, let $Y_{i, j}$ be the individual Bernoulli trials of
  $Y_i$, i.e., $Y_i = \sum_{j \in [n]} Y_{i, j}$.  For $j \in [n]$, we
  define the event $\binomial_j$ to be
  $Y_{1, j} + Y_{2, j} + Y_{3, j} \le 1$, and the event
  $\binomial = \bigcap_{j \in [n]} \binomial_j$.

  Based on this, we now define the random variables $X_1$, $X_2$,
  $X_3$, and $X_4$ as follows.  If $\binomial$ holds, we set
  $X_i = Y_i$ for $i \in [3]$ and $X_4 = n - X_1 - X_2 - X_3$.
  Otherwise, if $\overline\binomial$, we draw
  $X = (X_1, X_2, X_3, X_4) \sim \Multi\left(n,\boldsymbol{p}\right)$
  independently from $Y_1$, $Y_2$, and $Y_3$ with
  $\boldsymbol{p} = (p, q, q, 1 - p - 2q)$.  Note that $X$ clearly
  follows $\Multi\left(n,\boldsymbol{p}\right)$ if
  $\overline\binomial$.  Moreover, conditioned on $\binomial$, each
  individual trial increments exactly one of the variables $X_1$,
  $X_2$, $X_3$, or $X_4$ with probabilities $p$, $q$, $q$, and
  $1 - p - 2q$, respectively, i.e.,
  $X \sim \Multi\left(n,\boldsymbol{p}\right)$.

  Thus, we end up with $X \sim \Multi\left(n,\boldsymbol{p}\right)$.
  Additionally, we have three independent random variables
  $Y_1 \sim \Bin\left(n,p\right)$, and
  $Y_2, Y_3 \sim \Bin\left(n,q\right)$ with
  $(X_1, X_2, X_3) = (Y_1, Y_2, Y_3)$ if $\binomial$ holds.  Thus, to
  prove the lemma, it remains to show that
  $\Pr{\binomial} \ge 1 - 3n \max(p, q)^2$.  For
  $j \in [n]$, the probability that the $j$th trial goes wrong is
  \begin{align*}
    \Pr{\overline\binomial_j} 
    &= 1-\left((1 - p) (1-q)^2\right) - \left(p (1-q)^2\right) - 2\left(q (1-p) (1 - q)\right) \\
                  &= 2pq - 2pq^2 + q^2 \leq  2pq  + q^2  \leq 3 \cdot \max(p,q)^2.
  \end{align*}
  Using the union bound it follows that
  $
  \Pr{\overline\binomial} \leq \sum_{j \in [n]} \Pr{\overline\binomial_j } \leq 3 n \cdot \max(p,q)^2
  $.
\end{proof}

       \noindent As mentioned before, we are interested in the event $N_\uandv \ge N_\unov$ (and
        likewise $N_\uandv \ge N_\vnou$), in order to apply
        Theorem~\ref{thm:random-walk}.  Moreover, due to
        Lemma~\ref{lem:multi_vs_binom}, we know that $N_\uandv$ and
        $(N_\unov - 1)$ almost behave like independent random
        variables that follow $\Bin(n - 2, p)$ and $\Bin(n - 2, q)$,
        respectively.  The following lemma helps to bound the
        probability for $N_\uandv \ge N_\unov$.  Note that it gives a
        bound for the probability of achieving strict inequality
        (instead of just $\ge$), which accounts for the fact that
        $(N_\unov - 1)$ and not $N_\unov$ itself follows a binomial
        distribution.

	\begin{lemma}
        \label{lem:equality_of_binomial_distributions}
    	Let $n \in \N$ with $n \geq 2$, and let $p, q \in (0, \frac12]$ with $p \geq q$.
        Further, let $X \sim \Bin(n,p)$ and $Y \sim \Bin(n,q)$ be independent, let $d = \lfloor p(n + 1) \rfloor$, and let $d = \smallO{\sqrt{n}}$, then $\Pr{X > Y} \geq \big(\frac12 - 1/\sqrt{2 \circlePi d}\big) (1 - \smallO{1})$.
	\end{lemma}

	\begin{proof}
		By \Cref{lem:dominating_binomial_distribution}, we get $\Pr{X \geq Y} \geq \frac12$, and we bound
		\[
			\Pr{X > Y} = \Pr{X \geq Y} - \Pr{X = Y} \geq \frac12 -  \Pr{X = Y},
		\]
		leaving us to bound $\Pr{X = Y}$ from above.
        By independence of~$X$ and~$Y$, we get
		\begin{align}
            \label{eq:equal_binomial_variables}
			\Pr{X = Y} = \sum_{i \in [n]} \Pr{X = i} \cdot \Pr{Y = i}.
		\end{align}
        Note that, by \Cref{lem:maximum_of_a_binomial_distribution}, for all $i \in [0.. n] $, it holds that $\Pr{X = i} \leq \Pr{X = d}$.
        Assume that we have a bound $B$ such that $\Pr{X = d} \leq B$.
        Substituting this into \Cref{eq:equal_binomial_variables} yields
\[
            \Pr{X = Y} \leq B \sum_{i \in [n]} \Pr{Y = i} = B,
      \]
        resulting in $\Pr{X > Y} \geq \frac12 - B$.
        Thus, we now derive such a bound for~$B$ and apply the inequality that for all $x \in \R$, it holds that $1 + x \leq \eulerE^x$, as well as \Cref{eq:stirling_inequality_lowerbound}. We get
        
        \begin{align}
            \notag
            \binom{n}{d} p^d (1 - p)^{n - d} &\leq \frac{n^d}{d!} \left(\frac{d}{n}\right)^d \left(1 - \frac{d}{n}\right)^n \left(1 - \frac{d}{n}\right)^{-d}\\ \notag
            &\leq \frac{d^d}{d!} \eulerE^{-d} \left(1 - \frac{d}{n}\right)^{-d}\\ \notag
            &\leq \frac{d^d}{\sqrt{2\circlePi} d^{d + 1/2} \eulerE^{-d}} \eulerE^{-d} \left(1 - \frac{d}{n}\right)^{-d}\\
        \label{eq:equal_binomials_last_step}
            &= \frac{1}{\sqrt{2 \circlePi d}} \frac{1}{(1 - d/n)^d}.
        \end{align}
        By Bernoulli's inequality, we bound $(1 - d/n)^d \geq 1 - d^2/n = 1 - \smallO{1}$ by the assumption $d = \smallO{\sqrt{n}}$.
        Substituting this back into \Cref{eq:equal_binomials_last_step} concludes the proof.
    \end{proof}
Finally, in order to apply Theorem~\ref{thm:random-walk}, we have
to make sure not to end up in the special case where $a = b \le 1$,
i.e., we have to make sure that the common neighborhood includes at
least two vertices.  The probability for this to happen is given by
the following lemma.

\begin{lemma}
  \label{lem:prob-x-more-than-1}
  Let $X \sim \Bin(n, p)$ and let $c = pn \in \smallO{n}$.  Then it
  holds that
  $\Pr{X > 1} \ge \left(1 - \eulerE^{-c}\left(1 + c \right)\right)
  \left(1 - \smallO{1}\right)$.
\end{lemma}
\begin{proof}
  As $X > 1$ holds if and only if $X \not= 0$ and $X \not=1$, we get
  \begin{align*}
    \Pr{X > 1}
    &= 1 - \Pr{X = 0} - \Pr{X = 1}
    = 1 - (1 - p)^n - n \cdot p \cdot (1 - p)^{n - 1}.\\
    \intertext{Using that for all $x \in \R$ it holds that $1 - x \le \e^{-x}$, we get}
    \Pr{X > 1}
    &\ge 1 - \eulerE^{-pn} - n \cdot p \cdot \eulerE^{-p(n - 1)}\\
    &= 1 - \eulerE^{-c} - c \cdot \eulerE^{c/n} \cdot \eulerE^{-c}\\ 
    &= 1 - \eulerE^{-c}\left(1 + c \cdot \eulerE^{c/n}  \right).
  \end{align*}
  As $\eulerE^{c/n}$ goes to $1$ for $n\to\infty$, we get the claimed
  bound. 
\end{proof}

\subsection{Many Edges Have Large Common Regions}
\label{sec:many-edges-have}

In \Cref{subsec:neighborhood_region}, we derived a lower bound on the probability
that $N_\uandv \ge N_\unov$ provided that the probability for a
vertex to end up in the shared region $R_\uandv$ is sufficiently large~compared to $R_\unov$.  In the following, we estimate the
measures of these regions depending on the~distance between $u$ and
$v$.  Then, we give a lower bound on the
probability that $\mu(R_\uandv) \ge \mu(R_\unov)$.

	\begin{lemma}
		\label{lem:probability_of_px}
		Let $G \sim \mathcal{G}(n,r)$ be a random geometric graph with expected average degree $\avdegree$, let $\{u, v\} \in E$ be an edge, and let $\tau\coloneqq \frac{\dist(u, v)}{r}$. Then, 
		\begin{align}
			\label{eq:pa_computation} 
			\mu(R_\uandv) &= \frac{\avdegree}{(n-1)\circlePi} \left(2 \arccos \left(\frac{\tau}{2}\right) - \sin\left(2\arccos\left(\frac{\tau}{2}\right)\right)\right) \textrm{ and}\\
			\label{eq:pb_computation}
			\mu(R_\unov) = \mu(R_\vnou) &= \frac{\avdegree}{n-1} - \mu(R_\uandv).
		\end{align}
	\end{lemma}
		
		\begin{figure}
			\centering
			\begin{subfigure}[t]{0.22\textwidth}
				\begin{center}
					\includegraphics[width=0.95\textwidth]{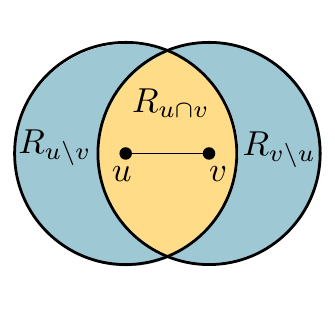}{\centering}
				\end{center}
				\subcaption{The geometric regions corresponding to the common and exclusive neighborhoods, respectively, with yellow illustrating $R_\uandv$ and blue illustrating $R_\unov$ and $R_\vnou$. }
				\label{fig:intersection_neighborhood}
			\end{subfigure} \hspace*{1em}
			\begin{subfigure}[t]{0.22\textwidth}
				\begin{center}
					\includegraphics[width=0.95\textwidth]{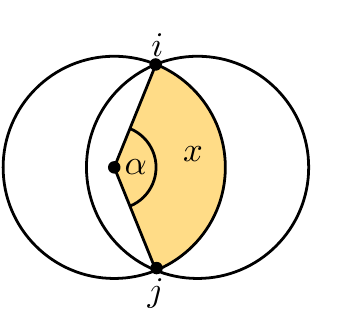}{\centering}
				\end{center}
				\subcaption{Let $\alpha$ be the central angle determined by the intersection points~$i$ and $j$, and let $x$ be the corresponding circular sector (illustrated in yellow). }
				\label{subfig:intersection_alpha}
			\end{subfigure} \hspace*{1em}
			\begin{subfigure}[t]{0.22\textwidth}
				\begin{center}
					\includegraphics[width=0.95\textwidth]{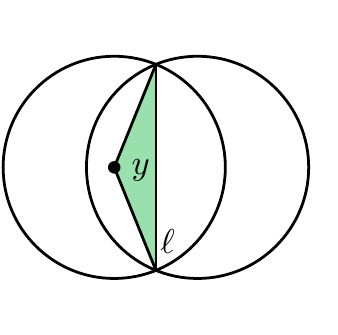}
				\end{center}
				\subcaption{Let $y$ be a triangle in the intersection (illustrated in green) determined by the radical axis $\ell$ and the central angle $\alpha$, cf. Fig~\ref{subfig:intersection_alpha}. }
				\label{subfig:intersection_trianglecircle}
			\end{subfigure}  \hspace*{1em}
			\begin{subfigure}[t]{0.22\textwidth}
				\begin{center}
					\includegraphics[width=0.375\textwidth]{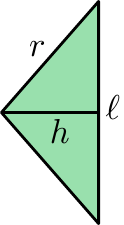}
				\end{center}
				\subcaption{The height~$h$ divides the area $\mu(y)$ (illustrated in green) of the triangle $y$, cf. Fig~\ref{subfig:intersection_trianglecircle}, into two subareas of equal size, since adjacent and opposite legs have the same length~$r$. }
				\label{subfig:intersection_triangle}
			\end{subfigure}
			\caption{The neighborhood of two adjacent vertices $u$ and $v$ in a random geometric graph.}
			\label{fig:intersection}
		\end{figure}
	
	\begin{proof}
		
		We start with proving \Cref{eq:pa_computation}. Let $i$ and $j$ be the two intersection points of the disks of $u$ and~$v$, let $\alpha$ be the central angle enclosed by $i$ and $j$, and let $x$ be the corresponding circular sector, cf. Fig~\ref{subfig:intersection_alpha}. Moreover, let the triangle~$y$ be a subarea of~$x$ determined by~$\alpha$ and the radical axis~$\ell$, cf. Fig~\ref{subfig:intersection_trianglecircle}. Let $h$ denote the height of the triangle $y$, cf. Fig~\ref{subfig:intersection_triangle}. For our calculations, we restrict the length of $\ell$ by the intersection points $i$ and $j$.
		Since we consider the intersection between disks and thus~$\ell$ divides the area $\mu(R_\uandv)$ into two subareas of equal sizes, it holds that $\mu(R_\uandv) = 2\left( \mu(x) - \mu(y)\right)$. Considering the two areas $\mu(x)$ and $\mu(y)$, it holds that
		\begin{align}
			\label{eq:area_a}
			\mu(x) = \frac{\alpha}{2}r^2 \quad \text{ and }\quad 	\mu(y) = h \cdot \frac{\ell}{2} = \cos \left(\frac{\alpha}{2}\right)r \cdot  \sin\left(\frac{\alpha}{2}\right)r = \frac{\sin(\alpha)}{2}r^2.
		\end{align}
		
		\noindent	For the central angle $\alpha$ we know $	\cos \left(\alpha/2\right) = h/r = \tau/2$ and therefore $\alpha = 2 \arccos \left(\frac{\tau}{2}\right)$. Together with \cref{eq:area_a}, we obtain
		\begin{align}
			\label{eq:pa}
			\mu(R_\uandv) = 2\left( \mu(x) - \mu(y)\right) = 2\ \left( \frac{2 \arccos \left(\frac{\tau}{2}\right)}{2}r^2 - \frac{\sin\left(2 \arccos \left(\frac{\tau}{2}\right)\right)}{2}r^2\right).
		\end{align}
		
		\noindent	The area of a general circle is equal to $\circlePi r^2$, the area of one disk in the random geometric graph equals $\frac{\avdegree}{n - 1}$, i.e., $ r^2 = \frac{\avdegree}{(n-1)\circlePi}$. Together with \cref{eq:pa}, we obtain \cref{eq:pa_computation}.

\Cref{eq:pb_computation}:  
		We get the claimed equality by noting that $\mu(R_\uandv) + \mu(R_\unov) = \circlePi r^2 $.	
	\end{proof}

	\begin{lemma}
          \label{lem:larger-shared-area-for-many-edges}
		Let $G \sim \mathcal{G}(n,r)$ be a random geometric graph, and let $\{u, v\} \in E$ be an edge. Then $\Pr{\mu(R_\uandv) \ge \mu(R_\unov)} \geq \left(\frac45\right)^2$.
	\end{lemma}

	\begin{proof}
		Let $\tau = \frac{\dist(u, v)}{r}$. By \Cref{lem:probability_of_px} with $\mu(R_\uandv) \geq \mu(R_\vnou)$, we get
		\[
			\left(2 \arccos \left(\frac{\tau}{2}\right) - \sin\left(2\arccos\left(\frac{\tau}{2}\right)\right)\right) \geq \frac{\circlePi}{2},
		\]
		which is true for $\tau \geq \frac45$. The area of a disk of radius $\frac45 r$ is $\left(\circlePi(\frac45r)^2\right)/\left(\circlePi r^2\right) = \left(\frac45\right)^2$ times the area of a disk of radius $r$. Hence, the fraction of edges with distance at most $\frac45r$ is at least~$\left(\frac45\right)^2$, concluding the proof.
	\end{proof}

\subsection{Proof of Theorem~\ref{thm:final-result}}

By Theorem~\ref{lem:decision_tree}, the probability that a random edge
$\{u, v\}$ is monochrome is at least $\frac12 + \Pr{\decisiveness}/2$,
where $\decisiveness$ is the event that the common neighborhood of $u$
and $v$ is more decisive than each exclusive neighborhood.  It remains
to bound $\Pr{\decisiveness}$.

\subparagraph{Existence of an edge yields a large shared region.}

Let $\region$ be the event that $\mu(R_\uandv) \ge \mu(R_\unov)$.
Note that this also implies $\mu(R_\uandv) \ge \mu(R_\vnou)$ as
$\mu(R_\unov) = \mu(R_\vnou)$.  Due to the law of total probability,
we have
\begin{equation*}
  \Pr{\decisiveness} \ge
  \Pr{\region} \cdot \Pr{\decisiveness \mid \region}.
\end{equation*}
Due to Lemma~\ref{lem:larger-shared-area-for-many-edges}, we have $\Pr{\region} \geq  \left(\frac45\right)^2$.  By conditioning on $\region$ in the
following, we can assume that
$\mu(R_\uandv) \ge \frac{\avdegree}{2n} \ge \mu(R_\unov) = \mu(R_\vnou)$,
where $\avdegree$ is the expected average degree.

\subparagraph{Neighborhood sizes are roughly binomially distributed.}

The next step is to go from the size of the regions to the number of
vertices in these regions.  Each of the remaining $n' = n - 2$
vertices is sampled independently to lie in one of the regions
$R_\uandv$, $R_\unov$, $R_\vnou$, or $R_\nouv$.  Denote the resulting
numbers of vertices with $X_1$, $X_2$, $X_3$, and $X_4$, respectively.
Then $(X_1, X_2, X_3, X_4)$ follows a multinomial distribution with
parameter $\boldsymbol{p} = \left(p, q, q, 1-p-2q\right)$ for
$p = \mu(R_\uandv)$ and $q = \mu(R_\unov) = \mu(R_\vnou)$.  Note that
$N_\uandv = X_1$, $N_\unov = X_2 + 1$, and $N_\vnou = X_3 + 1$ holds
for the sizes of the common and exclusive neighborhoods, where the
$+1$ comes from the fact that $v$ is always a neighbor of $u$ and vice
versa.

We apply Lemma~\ref{lem:multi_vs_binom} to obtain independent
binomially distributed random variables $Y_1$, $Y_2$, and $Y_3$ that
are likely to coincide with $X_1 = N_\uandv$, $X_2 = N_\unov - 1$, and
$X_3 = N_\vnou - 1$, respectively.  Let $\binomial$ denote the event
that $(N_\uandv, N_\unov - 1, N_\vnou - 1) = (Y_1, Y_2, Y_3)$.  Again,
using the law of total probabilities and due to the fact that
$\region$ and $\binomial$ are independent, we get
\begin{equation*}
  \Pr{\decisiveness \mid \region} \ge
  \Pr{\binomial \mid \region}\cdot
  \Pr{\decisiveness \mid \region \cap \binomial} =
  \Pr{\binomial}\cdot
  \Pr{\decisiveness \mid \region \cap \binomial}.
\end{equation*}
Note that $p, q \le \frac{\avdegree}{n}$ for the expected average degree
$\avdegree$.  Thus, Lemma~\ref{lem:multi_vs_binom} implies that
$\Pr{\binomial} \ge \left(1 - 3\avdegree^2/n\right)$.
Conditioning on $\binomial$ makes it correct to assume that
$N_\uandv \sim \Bin(n', p)$, $(N_\unov - 1) \sim \Bin(n', q)$,
$(N_\vnou - 1) \sim \Bin(n', q)$ are independently distributed.
Additionally conditioning on $\region$ gives us
$p \ge \frac{\avdegree}{2n} \ge q$.

\subparagraph{A large shared region yields a large shared neighborhood.}

In the next step, we consider an event that makes sure that the number
$N_\uandv$ of vertices in the shared neighborhood is sufficiently
large.  Let $\neighborhood_1$, $\neighborhood_2$, and
$\neighborhood_3$ be the events that $N_\uandv \ge N_\unov$,
$N_\uandv \ge N_\vnou$, and $N_\uandv > 1$, respectively.  Let
$\neighborhood$ be the intersection of $\neighborhood_1$,
$\neighborhood_2$, and $\neighborhood_3$.  We obtain
\begin{align*}
  \Pr{\decisiveness \mid \region \cap \binomial}
  &\ge \Pr{\neighborhood\mid  \region \cap \binomial} \cdot
    \Pr{\decisiveness \mid \region \cap \binomial \cap \neighborhood}\\
  &\ge \Pr{\neighborhood_1\mid  \region \cap \binomial} \cdot
    \Pr{\neighborhood_2\mid  \region \cap \binomial} \cdot
    \Pr{\neighborhood_3\mid  \region \cap \binomial} \cdot
    \Pr{\decisiveness \mid \region \cap \binomial \cap \neighborhood},
\end{align*}
where the last step follows from
Lemma~\ref{lem:obvious-observation-on-conditional-probabilities} as
the inequalities in $\neighborhood_1$, $\neighborhood_2$, and
$\neighborhood_3$ all go in the same direction.  Note that
$N_\uandv \ge N_\unov$ is equivalent to $N_\uandv > N_\unov - 1$.  Due
to the condition on $\binomial$, $N_\uandv$ and $N_\unov - 1$ are
independent random variables following $\Bin(n', p)$ and
$\Bin(n', q)$, respectively, with $p \ge q$ due to the condition on
$\region$.  Thus, we can apply
Lemma~\ref{lem:equality_of_binomial_distributions}, to obtain
\begin{equation*}
  \Pr{\neighborhood_1\mid \region \cap \binomial} =
  \Pr{\neighborhood_2\mid \region \cap \binomial} \ge
  \frac{1}{2} - \frac{1}{\sqrt{2 \pi \lfloor \avdegree / 2\rfloor}(1 -
    \smallO{1})},
\end{equation*}
and Lemma~\ref{lem:prob-x-more-than-1} gives the bound
\begin{equation*}
  \Pr{\neighborhood_3\mid \region \cap \binomial} \ge
  1 - \eulerE^{-\avdegree/2} \left(1 + \frac{\avdegree}{2}\cdot (1 + \smallO{1})\right).
\end{equation*}
Note that both of these probabilities are bounded away from $0$ for
$\avdegree \ge 2$.
Conditioning on $\neighborhood$ lets us assume that the shared
neighborhood of $u$ and $v$ contains at least two vertices and that it
is at least as big as each of the exclusive neighborhoods.

\subparagraph{A large shared neighborhood yields high decisiveness.}

The last step is to actually bound the remaining probability
$\Pr{\decisiveness \mid \region \cap \binomial \cap \neighborhood}$.
Note that, once we know the number of vertices in the shared and
exclusive neighborhoods, the decisiveness no longer depends on
$\region$ or $\binomial$, i.e., we can bound
$\Pr{\decisiveness \mid \neighborhood}$ instead. For this, let
$\decisiveness_1$ and $\decisiveness_2$ be the events that
$D_\uandv > D_\unov$ and $D_\uandv > D_\vnou$, respectively. Note
that $\decisiveness$ is their intersection.  Moreover, due to
Lemma~\ref{lem:obvious-observation-on-conditional-probabilities}, we
have
$\Pr{\decisiveness \mid \neighborhood} \ge \Pr{\decisiveness_1 \mid
  \neighborhood} \cdot \Pr{\decisiveness_2 \mid \neighborhood}$.  To
bound
$\Pr{\decisiveness_1 \mid \neighborhood} = \Pr{\decisiveness_2 \mid
  \neighborhood}$, we use Theorem~\ref{thm:random-walk}.  Note
that the $b$ and $a$ in Theorem~\ref{thm:random-walk}
correspond to $N_\uandv$ and $N_\unov + 1$ (the~$+ 1$ coming from the
fact that $N_\unov$ does not count the vertex $v$).  Moreover
conditioning on $\neighborhood$ implies that $a \le b$ and $b > 1$.
Thus, Theorem~\ref{thm:random-walk} implies
$\Pr{\decisiveness_1 \mid \neighborhood} \ge \frac{3}{16}$.

\subparagraph{Conclusion.}

The above arguments gives us that the fraction of monochrome edges
is
\begin{equation*}
  \frac{1}{2} + \frac{\Pr{\decisiveness}}{2}
  \ge
  \frac{1}{2} + \frac{1}{2} \cdot
  \underbrace{\Pr{\region}}_{\ge \left(\frac{4}{5}\right)^2} \cdot
  \underbrace{\Pr{\binomial}}_{1 - \smallO{1}} \cdot
  \big(
  \underbrace{\Pr{\neighborhood_1 \mid \region \cap \binomial}}_{
    \ge \frac{1}{2} - \frac{1}{\sqrt{2 \pi \lfloor \avdegree / 2\rfloor}}
  }
  \big)^2 \cdot
  \underbrace{\Pr{\neighborhood_3 \mid \region \cap \binomial}}_{
    \ge 1 - \e^{-\avdegree/2} \left(1 + \frac{\avdegree}{2}\right)
  } \cdot
  \big(
  \underbrace{\Pr{\decisiveness_1 \mid \neighborhood}}_{
    \ge \frac{3}{16}
  }
  \big)^2,
\end{equation*}
where we omitted the $\smallO{1}$ terms for
$\Pr{\neighborhood_1 \mid \region \cap \binomial}$ and
$\Pr{\neighborhood_3 \mid \region \cap \binomial}$, as they are
already covered by the $1 + \smallO{1}$ coming from $\Pr{\binomial}$.
This yields the bound stated in Theorem~\ref{thm:final-result}:
\begin{equation*}
  \frac{1}{2} +
  \frac{9}{800} \cdot
  \left(\frac{1}{2} - \frac{1}{\sqrt{2 \pi \lfloor \avdegree / 2\rfloor}}\right)^2 \cdot
  \left(1 - \eulerE^{-\avdegree/2} \left(1 + \frac{\avdegree}{2}\right)\right) \cdot
  (1 - \smallO{1}).
\end{equation*}

\section{Erd{\H o}s--Rényi Graphs}\label{sec:erdosrenyi}
	In the following, we are interested in the probability that an edge $\{u, v\}$ is monochrome after the \shortmodel\ on Erd{\H o}s--Rényi graphs. In contrast to geometric random graphs, we prove an upper bound. To this end, we show that it is likely that the common neighborhood is empty and therefore $u$ and $v$ choose their types to be the predominant type in their exclusive neighborhood, which is $t^+$ and $t^-$ with probability $\frac12$, each.
	
	\begin{theorem}
		\label{Thm:ER}
		Let $G \sim \mathcal{G}(n,p)$ be an Erd{\H o}s--Rényi graph with expected average degree $\avdegree = \smallO{\sqrt{n}}$. The expected fraction of monochrome edges after the FSP is at most $\frac12 + \smallO{1}$.
	\end{theorem}

\begin{proof}
Given an edge $\{u,v\}$, let $\monochrome$ be the event that $\{u,v\}$ is monochrome. We first split~$\monochrome$ into disjoint sets with respect to the size of the common neighborhood and apply the law of total probability and get
	\begin{align*}
		\Pr{\monochrome} 	=\ & \Pr{\monochrome}[N_\uandv= 0] \cdot \Pr{N_\uandv= 0} + \Pr{\monochrome}[N_\uandv> 0] \cdot \Pr{N_\uandv> 0} \\
		 \leq\ & \Pr{\monochrome}[N_\uandv= 0] \cdot 1 + 1 \cdot \Pr{N_\uandv> 0} .
	\end{align*}
We bound each of the summands separately. 
For estimating $\Pr{\monochrome}[N_\uandv= 0]$, we note that the types of $u$ and~$v$ are determined by the predominant type in disjoint vertex sets.~By definition of the \shortmodel\ this implies that the probability of an monochrome edge is equal to~$\frac12$.

We are left with bounding $\Pr{N_\uandv > 0}$. Note that $N_\uandv \sim \Bin\left(n,p^2\right)$.
	Thus, by Bernoulli's inequality we get $\Pr{N_\uandv > 0} = 1 - \Pr{N_\uandv = 0} = 1 - \left(1-p^2\right)^n  \leq np^2$. Noting that $np^2 = \smallO{1}$ holds due to our assumption on $\avdegree$, concludes the proof.
\end{proof}

\newpage

\bibliography{References}

\end{document}